\documentclass[11pt,twoside]{article}
\usepackage{a4wide}
\usepackage{amssymb}
\usepackage{epsfig,enumerate,amsmath,amsfonts,amssymb}
\usepackage{indentfirst}
\usepackage{pstricks}
\usepackage{setspace}
\usepackage{latexsym}
\usepackage{amsthm}
\usepackage{amsmath}
\usepackage{enumerate}

\newcommand\myText[1]{\text{\scriptsize\tabular[t]{@{}l@{}}#1\endtabular}}
\newtheorem{defi}{Definition}[section]
\newtheorem{thm}{Theorem}[section]
\newtheorem{rem}{Remark}[section]
\newtheorem{lem}{Lemma}[section]

\newtheorem{cor}{Corollary}[section]
\newcommand{\ds} {\displaystyle}
\newcommand{\e}{\epsilon}

\newcommand{\al} {\alpha}
\newcommand{\ba} {\beta}

\newcommand{\ra} {\rightarrow}

\newcommand{\De} {\Delta}
\newcommand{\la} {\lambda}

\newcommand{\noi} {\noindent}
\newcommand{\na} {\nabla}

\newcommand{\mb} {\mathbb}
\newcommand{\mc} {\mathcal}

\usepackage[all]{xy}
\catcode`\@=11
\def\theequation{\@arabic{\c@section}.\@arabic{\c@equation}}
\catcode`\@=12

\def\R{{I\!\!R}}

\begin{document}

\title
{Critical exponent problems for $\frac{1}{2}$-Laplacian in $\mathbb R$}
\date{}
\maketitle
\author{
\begin{center}
{\sc J.  Giacomoni}\\
   LMAP (UMR CNRS 5142) Bat. IPRA,\\
  Avenue de l'Universit\'e \\
   F-64013 Pau, France\\
{\it e-mail: jacques.giacomoni@univ-pau.fr}\\
\vspace{1cm}
{\sc Pawan Kumar Mishra} and {\sc K. Sreenadh}\\
Department of Mathematics, \\ Indian Institute of Technology Delhi\\
Hauz Khaz, New Delhi-16, India
\\
{\it e-mail: pawanmishra31284@gmail.com, sreenadh@gmail.com}
\end{center}
}

%{\bf  Pawan Kumar Mishra\footnote{email: pawanmishra31284@gmail.com}} and {\bf  K. Sreenadh\footnote{e-mail: sreenadh@gmail.com}}\\
%{\small Department of Mathematics}, \\{\small Indian Institute of Technology Delhi}\\
%{\small Hauz Khaz}, {\small New Delhi-16, India.}\\
% }

\date{}

\maketitle

\begin{abstract}
\noindent We study the existence  of {weak} solutions for  fractional elliptic equations of the type,
\begin{equation*}
(-\Delta)^{\frac{1}{2}} u+ V(x) u= h(u) , u> 0 \;\textrm{in } \;\mathbb R,
\end{equation*}
%where $1<q<2,\;p>2,\;1<\beta\leq2\;, \lambda>0, K(x)>0, f$ is continuous and sign changing.
where $h$ is a real valued function that behaves like $e^{u^2}$ as $u\rightarrow \infty$ and $V(x)$ is a positive, continuous unbounded function.
Here $(-\Delta)^{\frac{1}{2}}$ is the fractional Laplacian operator. We show the existence of mountain-pass solution when the nonlinearity is superlinear near $t=0$. We also study the corresponding critical exponent problem for the Kirchhoff equation
\[  m\left(\int_{\mathbb R}|(-\Delta)^{\frac{1}{2}}u|^2 dx+ \int_\mb R u^2 V(x)dx\right)\left((-\Delta)^{\frac{1}{2}} u+ V(x) u\right)= f(u)\;\; \text{in}\; \mathbb R   \]
 where $f(u)$ behaves like $e^{u^2}$ as $u\rightarrow \infty$ and $f(u)\sim u^3$ as $u\rightarrow 0$.
\end{abstract}

\section{Introduction}
\noindent
In this article, we study the existence  of {weak} solutions for  fractional elliptic equations of the type,
\begin{equation*}
(P)\; \quad(-\Delta)^{\frac{1}{2}} u+ V(x) u= h(u) , u> 0 \;\textrm{in } \;\mathbb R,
\end{equation*}
%where $1<q<2,\;p>2,\;1<\beta\leq2\;, \lambda>0, K(x)>0, f$ is continuous and sign changing.
where the nonlinearity $h(u)$ satisfies critical growth of exponential type  which will be stated later.

\noindent We also study the corresponding critical exponent problem for the Kirchhoff equation
\[ (Q)\quad \left\{ m\left(\int_{\mathbb R}|(-\Delta)^{\frac{1}{2}}u|^2 dx+ \int_\mb R u^2 V(x)dx\right)\left((-\Delta)^{\frac{1}{2}} u+ V(x) u\right)= f(u)\;\; \text{in}\; \mathbb R\right.\] where $m:\mathbb R^+ \rightarrow \mathbb R^+$ and $f:\mathbb R \rightarrow \mathbb R^+$ are continuous functions that satisfy some {suitable} conditions. \\

\noindent The function $V(x)$ is a continuous function satisfying the following assumption:\\
\textbf{(V)} $ V(x)\ge V_0>0 $ in $\mathbb R$ and $V(x)\rightarrow \infty$ as $|x|\rightarrow \infty$. \\
An example of function satisfying the above assumption is $V(x)=\vert x\vert^p+V_0$ with $p>0$ and $V_0>0$.

\noindent  Here $(-\Delta)^{\frac{1}{2}}$ is the $\frac{1}{2}$-Laplacian operator defined as
\begin{equation*}
 (-\Delta)^{\frac{1}{2}}u(x)=\int_{\mathbb{R}} \frac{(u(x+y)+u(x-y)-2u(x))}{|y|^{2}}dy\;\;\;\;\;\textrm{for all}\; \; x\in \mathbb{R}.
 \end{equation*}
  The fractional Laplacian  operator has been a classical topic in Fourier analysis and nonlinear partial differential equations for a long time.  Fractional operators are  involved in financial mathematics, where Levy processes with jumps appear in modeling the asset prices (see \cite{app}).
Recently the fractional Laplacian has attracted many researchers. In particular, concerning nonlinear elliptic equations involving fractional operators, the issues of existence and properties of solutions (regularity, a priori bounds, asymptotic behavior, symmetry, etc.) have been discussed in detail (see for instance \cite{XcT}, \cite{CS}, \cite{ChKiLe}, \cite{JPS}, \cite{ionzsq}, \cite{PuTe}, \cite{secchi} and \cite{YYY}. The critical exponent problems for square root of Laplacian are studied in \cite {capella}, \cite{tan}.

\noindent In \cite{tan}, authors have studied the following Brezis-Nirenberg type critical exponent problem on bounded domains $\Omega \subset \mathbb R^n, \; n\ge 2$:
\begin{equation*}
(-\Delta)^{\frac{1}{2}} u
= \lambda u+ u^{\frac{2n}{n-1}},\; u>0\;  \text{in } \Omega,\quad
u=0 \;\mbox{in }  \mb{R}^n\backslash \Omega,
\end{equation*}
 by studying its harmonic extension problem. The idea of these harmonic extensions was initially introduced and studied in the beautiful work of Caffarelli and Silvestre \cite{CS}. The critical exponent problems in the limiting case $n=1$ and with nonlinearities with exponential growth are studied in \cite{JPS}. Here the exponential type nonlinearity is motivated by fractional Moser-Trudinger embedding due to Ozawa \cite{yz, ozaw}.
 %Problems of this type in local settings were studied in \cite{fmr, doo}.

\noindent In \cite {mrsq}, authors considered the problem in the whole space $\mathbb R$:
\begin{equation*}
(-\Delta)^{1/2}u + u=K(x)g(u) \;\;\textrm{in}\;\; \mathbb{R}
\end{equation*} where $K$ is a real valued positive function and $g$ {has a} critical exponential growth and {is} super-quadratic near $0$. Here authors proved the existence of solutions by studying the corresponding harmonic extension problem under suitable conditions on $K$ and $g$. In section {3}, we improve this result by identifying more accurately the first critical level under which the Palais-Smale condition holds. To achieve this, from the sharp Trudinger-Moser inequality of Theorem \ref{fmt} for bounded domains, we  derive a new Trudinger-Moser inequality in the whole space (see Theorem \ref{thmnew}) and show the existence of { a Palais-Smale} sequence that concentrates on the boundary $\mathbb{R}\times \{0\}$ in the spirit of \cite{adyadav} {and whose energy level is strictly below the first critical level}. {We highlight that the assumption (h4) (see section 3) plays an important role in proving such compactness of the exhibited Palais-Smale sequence. Furthermore, in the local setting (see \cite{adyadav},\cite{AdPr}, \cite{jacprsre}), (h4) appears to be the sharp condition on the asymptotic behaviour of nonlinearity $h$ to ensure the existence of nontrivial solutions for critical problems in two dimensions. We show that it still holds for more general non local problems as $(Q)$ investigated in section 4. }

\noindent Elliptic problems with exponential growth nonlinearities are motivated by the Moser-Trudinger inequality \cite{moserr}, namely
\[\sup_{\|u\|_{H^{1}_{0}(\Omega)}\le 1} \int_{\Omega} e^{\alpha u^2} dx <\infty, \; \text{if and only if} \; \alpha \le 4\pi, \]
where $\Omega \subset \R^2$ {is a} bounded domain. The existence of solutions for critical exponent problem was initiated and studied in \cite{adi, adyadav,fmr}. Subsequently, these results were generalized to unbounded domains in \cite{doo,panda}.

\noindent The space $H^{\frac{1}{2}}(\mathbb R)$ is the Hilbert space with the norm defined as
\begin{equation}
\|u\|^2= \|u\|_{L^2({\mathbb R})}+\int_{\mathbb R}|(-\Delta)^{\frac{1}{4}}u|^2dx.
\end{equation}
The space $H_0^{\frac{1}{2}}({\mathbb R})$ is the completion of $C_0^{\infty}({\mathbb R})$ under
$[u]=\left(\int_\mathbb R|(-\Delta)^\frac{1}{4}u|^2dx\right)^{\frac{1}{2}}$.

\noindent The problems of the type $(P)$ with exponential growth nonlinearities are motivated from the fractional Trudinger-Moser inequality \cite{lm}, which gives the optimal constant and improves the former results of Ozawa \cite{ozaw} and Kozono, Sato-Wadade \cite{ksw}. {Precisely, let} $I$ be a bounded interval of ${\mathbb R}$. Set $X(I):=\{u\in H^{\frac{1}{2}}({\mathbb R}) \,:\, u\equiv\,0 \mbox{ in }{\mathbb R}\backslash I\}$. {Then,}
\begin{thm}\label{fmt}
For $u\in H^{\frac{1}{2}}(I)$, $e^{\beta u^2} \in L^1(I)$ for any $\beta>0$. Moreover there exists a constant $C>0$ such that
\begin{eqnarray*}
\displaystyle\sup_{u\in X(I),\|(-\Delta)^{\frac{1}{4}} u\|_{
L^2(I)}\le 1 }\int_{I} e^{\alpha u^2} dx \le C |I|\quad \text{for all}\;\; \alpha \le \pi.
\end{eqnarray*}
\end{thm}
\noindent Our approach in the present paper is based on the Caffarelli-Silvestre approach to fractional Laplacians in \cite{CS}. In \cite{CS} it was shown that for any $v\in H^{\frac{1}{2}}(\mathbb{R}), $ the unique function $w(x,y)$ that minimizes the weighted integral
\[\mathcal{E}_{\frac{1}{2}}(w)=\int_{0}^{\infty}\int_{\mathbb{R}} |\nabla w(x,y)|^2 dx dy\]
over the set $\left\{w(x,y): \mathcal{E}_{\frac{1}{2}}(w)<\infty, w\vline_{y=0}=v\right\}$ satisfies
$\int_{\mathbb{R}}|(-\Delta)^{\frac{1}{{4}}}v|^2 = \mathcal{E}_{\frac{1}{2}}(w).$ Moreover $w(x,y)$ solves the boundary value problem
\begin{equation}
-\text{div}(\nabla w)=0 \; \text{in}\; \mathbb{R}\times \mathbb{R}_+,\;\; w\vline_{y=0}=v\quad
\frac{\partial w}{\partial \nu}=(-\Delta)^{1/2} v(x) \label{new1} \end{equation}
where $\frac{\partial w}{\partial \nu}=\displaystyle\lim_{y\rightarrow 0^+}\frac{\partial w}{\partial y}(x, y)$.
We denote the upper half space in $\mathbb{R}^2$ as $\mathbb{R}_+^2=\{(x, y)\in \mathbb{R}^2 |\; y>0\}$.
The space $X_1(\mathbb R^2_+)$ is defined as the completion of $C_0^{\infty}(\mathbb R^2_+)$ under the semi-norm
\begin{equation}
\|w\|_{X_1}=\left(\int_{\mathbb R^2_+} |\nabla w|^2~dxdy\right)^{\frac{1}{2}}.
\end{equation}
For a function $u\in H_0^{\frac{1}{2}}(\mathbb R)$, the solution $w\in X_1(\mathbb {R}^2_+)$ of the problem $\eqref{new1}$ is called harmonic extension of $u$. The map $\mathcal{E}_{\frac{1}{2}}: H_0^{\frac{1}{2}}(\mathbb R)\rightarrow X_1(\mathbb {R}^2_+)$ is an isometry. We look for solutions in the Hilbert space $E_V$ defined as
\begin{equation*}
E_V=\left\{w\in X_1(\mathbb R^2_+): \int_\mathbb R|w(x, 0)|^2 V(x) dx<\infty\right\}
\end{equation*}
equipped with the norm
\[\|w\|=\left(\int_{\mathbb R^2_+} |\nabla w|^2 dx dy + \int_{\mathbb R}|w(x,0)|^2 V(x) dx\right)^{\frac{1}{2}}.\]
\noindent From the assumption ${\bf (V)}$ and the continuous embedding of $E_V$ in $L^q({\mathbb R})$ for $q\in [2,\infty)$, it follows that the  embedding $ E_V \ni u \mapsto  |u|^r  \in L^1(\mathbb R)$ is compact for all $r\in [2,\infty)$ . Moreover, the minimization problem
\[\lambda_1 = \min_{w\in E_V} \left\{ \int_{\mathbb R^{2}_{+}} |\nabla w|^2+ \int_{\mathbb R} |w(x,0)|^2 V(x)dx: \int_{\mathbb R} |w(x,0)|^2 dx=1\right\}\]  admits a non-negative minimizer and $\lambda_1 \ge V_0>0.$
%
%The assumption $(K)$ implies that the trace embedding $E$ into  $L^2(\R)$ is compact. Furthermore, the smallest eigenvalue $\lambda_1$ is defined as
%\begin{equation}\label{eigen}
%\lambda_1=\min \frac{ \int_{\mathbb R^2_+} |\nabla w|^2 dx dy }{  \int_\R |w(x,0)|^2 K(x) dx }.
%\end{equation}
%is positive and is achieved by a non-negative eigenfunction $\phi_1 \in E$.
%
\noindent In this paper, first we discuss the Adimurthi \cite{adi} type existence result for the fractional Laplacian equation in $(P)$  with nonlinearity $ h(u)$ that has superlinear growth near zero and {critical} exponential
growth near $\infty$. To prove our result we analyze the first critical level using the Moser functions which are dilations and truncations of fundamental solutions in ${\mathbb R}^2$ and study the compactness of Palais-Smale sequences below this level.
In the second part, we discuss the Kirchhoff fractional Laplacian equation in $(Q)$ with critical exponential nonlinearity that behaves like $u^3$ near the origin and $e^{u^2}$ at $\infty.$ Here, using the critical level obtained in the section 2, we study the critical level for the Kirchhoff problems and we use the Moser functions concentrating on the boundary along with Lion's Lemma on higher integrability to show the strong convergence of Palais-Smale sequences below the critical level.
%with sign-changing and exponential type nonlinearity to obtain the multiplicity
%of solutions with respect to the parameter $\lambda$. We show the multiplicity result by
%extracting Palais-Smale sequences in the Nehari manifold. In the critical case, we translate the functional and show the existence of mountain-pass solution. \\

\noindent {We now give the organization of the paper.} In section 2, we present a version of Moser-Trudinger inequality which is the central idea of the proof of existence result in section 3. In section 3, we consider the critical exponent problem with positive
nonlinearity and prove Adimurthi's type existence result. In section 4, we consider the problem $(Q)$ and study the existence result.
\section{A Moser-Trudinger inequality}
\setcounter{equation}{0}
\noindent In this section we prove the Moser-Trudinger inequality on $\mathbb R^{2}_{+}$. We  follow the ideas
 in \cite{yz}. We prove the following theorem.
\begin{thm}\label{thmnew}
There exists a constant $C>0$ such that
\begin{equation}\label{fmtnew2}
\sup_{\|w\|_{H^1(\mathbb R^{2}_{+})}\le 1} \int_{\mathbb R} \left(e^{\alpha (w(x,0))^2} -1\right) dx \le C \; \text{for all}\;\; 0< \alpha < \pi.
\end{equation}
\end{thm}
 \noindent We define $I_R(x)=[x-R, x+R]$ and  $\mathcal{C}_{R}(x)= I_R(x) \times \mathbb R_{+}$ and the space $H^{1}_{0,L}(\mathcal{C}_{R}(x))$ as
\[H^{1}_{0,L}(\mathcal{C}_R(x))=\left\{ w\in H^{1}(\mathcal{C}_R(x)): w=0 \; \text{on} \; \{x-R, x+R\}\times \mathbb R_+\right\}.\]
Then $H^{1}_{0,L}(\mathcal{C}_{R}(x)) $ is a Hilbert space with respect to the norm
\[\|w\|_{\mathcal C} = \left(\int_{\mathcal{C}_{R}(x)} |\nabla w|^2\right)^{\frac{1}{2}}.\]
The following Moser-Trudinger trace inequality for $H^{1}_{0,L}(\mathcal{C}(x))$ follows from theorem \ref{fmt}:\\
There exists $C>0$  such that
\begin{equation}\label{fmtnew1}
\displaystyle \sup_{\|w\|_{\mathcal{C}_R(x)} \leq1} \int_{I_R(x)} e^{\alpha w(x, 0)^2}dx\leq C R.
\end{equation}
We have the following Lemma.
\begin{lem}
Let $0<\alpha<\pi$. For $x_0\in \mathbb R$ and $w\in H^{1}_{0,L}(\mathcal{C}_{R}(x_0))$ with $\|w\|_{\mathcal{C}} \le 1$, we have
\[
\int_{I_R(x_0)}\left(e^{\alpha w(x,0)^2}-1\right) dx \le C R \left(\int_{\mathcal{C}_R(x_0)} |\nabla w(x,y)|^2 dx dy \right).\]
\end{lem}
\begin{proof}
By \eqref{fmtnew1}, we have
\begin{equation}\label{new2.0}
\sup_{ \|w\|_{\mathcal{C}} \le 1} \int_{I_R(x_0)} (e^{\alpha (w(x,0))^2} -1) dx \leq C R, \; \text{for} \; \alpha\le\pi.
\end{equation}
Let $\tilde{w}=\frac{w}{\|w\|_{\mathcal{C}}}$ for any $w\in H^{1}_{0,L}(\mathcal{C}_R(x_0))\backslash \{0\}$. Then

\begin{align}\label{new2.1}
\int_{I_R(x_0)} (e^{\alpha {{\tilde{w}(x,0)}}^2} -1)dx & \ge \frac{1}{\| w\|_{\mathcal{C}}} \int_{I_R(x_0)} \sum_{k=1}^{\infty} \frac{\alpha^k |w|^{2k}}{k!} dx \nonumber\\
&= \frac{1}{\|w\|_{\mathcal{C}}} \int_{I_R(x_0)} (e^{\alpha (w(x,0))^2} -1) dx.
\end{align}
Now the proof follows from  \eqref{new2.0} and \eqref{new2.1}.
\end{proof}
\noindent
%Let us define $\mathcal{C}_R (x_i)= [x_i -R, x_i+R]\times \mathbb R_{+}$. Then
For any $R>0$, there exists $\{x_i\}\subset \mathbb R$ such that
\begin{enumerate}
\item $\cup_{i=1}^{\infty} \mathcal{C}_{\frac{R}{2}} (x_i) = \mathbb R^{2}_{+}$,
\item for $i\ne j,\; \mathcal{C}_{\frac{R}{4}}(x_i)\cap \mathcal{C}_{\frac{R}{4}}(x_j)=\emptyset$,
\item $\forall x\in \mathbb R^{2}_{+}$, $x$ belongs to at most $N$ sets $\mathcal{C}_R(x_i)$ for some $N$.
\end{enumerate}
Now we will prove the Theorem:\\
{\bf Proof of theorem \ref{thmnew}:} Let $R>0$ { to be fixed later} and let $\phi_i$ be cut-off { functions} such that
$(1)\; \phi_i\in {C^{\infty}}(\mathcal{C}_R(x_i)),$ { such that $\mbox{supp}(\phi_i)\subset K_i\times {\mathbb R}^+$ with  $K_I$ compact set of $(x_i-R,x_i+R)$},
$(2)\; 0\le \phi_i\le 1 \; \text{on}\; \mathcal{C}_R(x_i)$
 $(3)\; \phi\equiv 1\; \text{on}\; C_{\frac{R}{2}}$
 $(4)\; |\nabla \phi_i|\le \frac{4}{R}$.
For $w\in H^1(\mathbb R^{2}_{+})$ such that $\displaystyle \int_{\mathbb R^{2}_{+} }\left( |\nabla w|^2 + |w|^2\right)  \le 1,$ we have $\phi_i w \in H^{1}_{0,L}(\mathcal{C}_R(x_i))$. Moreover,
\begin{align}\label{new1.3}
\int_{\mathcal{C}_R(x_i)} |\nabla (\phi_i w)|^2 & \le (1+\epsilon) \int_{\mathcal{C}_R(x_i)}  \phi_{i}^{2}|\nabla  w|^2 +\frac{C(\epsilon)}{R^2} \int_{\mathcal{C}_R(x_i)} |w|^2\nonumber\\
&\le  (1+\epsilon)\left(\int_{\mathcal{C}_R(x_i)} |\nabla  w|^2 + \int_{\mathcal{C}_R(x_i)} |w|^2\right)
\end{align}
where in the last inequality we choose a sufficiently large $R$ to make sure $\frac{C(\epsilon)}{R}\le 1+\epsilon$. Let $\alpha_\epsilon= \frac{\pi}{1+\epsilon}$ and $\widetilde{\phi_i w} {:=\frac{\phi_i w}{(1+\epsilon)^{\frac{1}{2}}}}\in H^{1}_{0,L}(\mathcal{C}_R), $ by Lemma 2.1 and \eqref{new1.3}, we get
\begin{align}\label{new2.4}
{\int_{I_{\frac{R}{2}(x_i)} }\left(e^{\alpha_\epsilon (w(x,0))^2}-1\right) dx} &\le \int_{I_R(x_i)} \left(e^{\alpha_\epsilon |\phi_i w(x,0)|^2} -1\right)dx\notag\\
&=\int_{I_R(x_i)} \left(e^{\pi {|\widetilde{\phi_i w}|^2}}-1\right)dx\notag\\
&\le C R \int_{\mathcal{C}_R(x_i)} |\nabla(\widetilde{\phi_i w})|^2 dx dy \notag\\
&\le C R \int_{\mathcal{C}_R(x_i)} \left(|\nabla w|^2 + w^2\right) dxdy.
\end{align}
Therefore, by Lemma 2.1 and \eqref{new2.4}, we have
\begin{align*}
\int_{\mathbb R} \left(e^{\alpha_\epsilon (w(x,0))^2}-1\right) dx &\le \int_{\cup_{i=1}^{\infty} I_{\frac{R}{2}}(x_i)} \left(e^{\alpha_\epsilon (w(x,0))^2} -1\right)dx\\
&\le \sum_{i=1}^{\infty} \int_{I_{\frac{R}{2}}(x_i)} \left(e^{\alpha_\epsilon (w(x,0))^2}-1\right) dx\\
&\le \sum_{i=1}^{\infty}  C R \int_{\mathcal{C}_R(x_i)} (|\nabla w|^2+w^2) dx dy\\
&\le C R N \int_{\mathbb{R}^{2}_{+}} (|\nabla w|^2+w^2)dx\\
&\le C R N.
\end{align*}
For any $\alpha<\pi,$ we can choose $\epsilon>0$ sufficiently small such that $\alpha <\alpha_\epsilon$. This completes the proof of Theorem 1.1.
\begin{cor}
There exists a constant $C>0$ such that
\begin{equation}\label{newcor}
\sup_{\|w\|\le 1} \int_{\mb R} e^{\al |w(x,0)|^2} dx \le C, \quad \text{for all} \;\; 0<\al<\pi.
\end{equation}
\end{cor}
\begin{proof}
Proof follows from the fact that $\|w\|$ is an equivalent norm in $H^{1}(\mb R^{2}_{+})$ and Theorem 1.1.
\end{proof}
We have the following lemma
\begin{lem}\label{orlicz-property}
For any $u\in H^{\frac{1}{2}}(\mathbb R)$ and any $\alpha>0$, we have $\int_{\mathbb R} (e^{\alpha u^2}-1) dx<\infty$.
\end{lem}
\begin{proof}
 Let $u\in H^{\frac{1}{2}}(\mathbb R)$ and $\alpha>0$. From Proposition 1 in \cite{ozaw} page 261, there exists $M>0$ such that for any $q\geq 2$ and any $f\in H^{\frac{1}{2}}(\mathbb R)$,
\begin{eqnarray}\label{ozawa-ineq}
\Vert f\Vert_{L^q(\mathbb R)}\leq M q^{1/2}\Vert(-\Delta)^{1/4}f\Vert^{1-2/q}_{L^2(\mathbb R)}\Vert f\Vert^{2/q}_{L^2(\mathbb R)}.
\end{eqnarray}
Therefore, for $k\geq 1$
\begin{eqnarray*}
\Vert u\Vert_{L^{2k}(\mathbb R)}^{2k}\leq M^{2k}(2k)^k\Vert (-\Delta)^{1/4}u\Vert^{2k-2}_2\Vert u\Vert_{L^2(\mathbb R)}^2.
\end{eqnarray*}
Hence,
\begin{eqnarray*}
\int_{\mathbb R}(e^{\alpha u^2}-1)dx=\displaystyle\sum_{k=1}^\infty\frac{\alpha^k\Vert u\Vert_{L^{2k}(\mathbb R)}^{2k}}{(2k)!}\leq \displaystyle\sum_{k=1}^{\infty}\frac{\alpha^k}{(2k)!} M^{2k}(2k)^k\Vert(-\Delta)^{1/4}u\Vert^{2k-2}_2\Vert u\Vert_{L^2(\mathbb R)}^2
\end{eqnarray*}
which is a convergent sequence. This ends the proof of the lemma.\end{proof}
\section{Critical exponent problem}
\setcounter{equation}{0}
\noindent In this section we consider existence of solution for the problem
\begin{equation*}
(P)\quad \; (-\Delta)^{1/2}u + V(x) u= h(u) \;\;\textrm{in}\;\; \mathbb{R}
\end{equation*}
where $V(x)$ satisfies the assumption ${\bf (V)}$ and $h(u)$ satisfies the following critical growth conditions:
\begin{enumerate}
\item[(h1)] $h\in C^1(\mathbb R), h(t)=0 $ for $t\le 0$, $h(t)>0$ for $t>0$  and $h$ satisfies for any $\varepsilon>0$, $\displaystyle\lim_{t\to\infty} h(t)e^{-(1+\varepsilon) t^2}=0. $
\item[(h2)] There exists $\mu>2$ such that for all $u>0$,
\[0\le \mu H(u) \le uh(u),\;\; \text{where}\;\; H(u)=\int_{0}^{u} h(s) ds.\]
\item[(h3)]There exist positive  constants $t_0$, $M$ such that
\[  H(t)\le M h(t)\;\mbox{for all}\; t\in [t_0,+\infty).\]
\item[(h4)] $\displaystyle \lim_{t\rightarrow \infty} t h(t)e^{-t^2}=\infty.$
\item[(h5)]$\displaystyle \limsup_{u\rightarrow 0}\frac{2 H(u)}{u^2}<\lambda_1.$
%\; \text{where}\; \lambda_1=\min\{\|w\|^2: \int_{\mathbb R} |w(x,0)|^2=1\}>V_0>0.$
\end{enumerate}

\begin{rem}\label{examples}
A prototype examples of $h$ satisfying (h1)-(h5) are $t^pe^{t^2}$ with $p>1$ and $t^p(e^{t^\beta}-1)e^{t^2}$ with $p>1$ and $\beta\in (0, 2)$. Nonlinearities of the form $t^pe^{\beta t^2}$ ($\beta>0$, $p>1$) can be also dealt with according modifications in the assumptions and  minor changes in the proofs. Note in this case that the first critical level of the energy functional $I$ is $\frac{\pi}{2\beta}$.
\end{rem}
\noi The variational functional associated to the problem $(P)$ is given as
\begin{equation}\label{afpl}
\tilde{I}(u)=\frac{1}{2}\int_{\mathbb {R}}|(-\Delta)^{\frac{1}{4}}u|^2dx+\frac{1}{2}\int_{\mathbb R}|u|^2 V(x) dx-\int_{\mathbb {R}}H(u)dx.
\end{equation}
%In \cite{XcT, CS}, authors developed local interpretation of factional operator by considering Neumann type operator in $\mathbb{R}^2_+$.
%The space $X_1(\mathbb R^2_+)$ is defined as the completion of $C_0^{\infty}(\mathbb R^2_+)$ under the norm
%\begin{equation}
%\|w\|_{X_1}=\left(\int_{\mathbb R^2_+} |\nabla w|^2dx\right)^{\frac{1}{2}}
%\end{equation}
The harmonic extension problem corresponding to $(P)$ is
\begin{equation*}\label{aE}
(P_E)\;
\left\{\begin{array}{rl}
 -div(\nabla w)&=0,\quad w>0  \quad\textrm{in}\; \mathbb {R}^2_+,  \\
  \frac{\partial w}{\partial y}&= -w(x,0) V(x)+ h(w(x,0))\;\; \textrm{on}\;\; \mathbb R.
\end{array}
\right.
\end{equation*}
The variational functional, $I:E_V\rightarrow \mathbb R$ related to the problem $(P_E)$ is given as
\begin{equation}\label{afel}
I(w)=\frac{1}{2}\int_{\mathbb {R}^2_+}|\nabla w|^2dxdy+\frac{1}{2}\int_{\mathbb {R}}|w(x,0)|^2 V(x) dx-\int_{\mathbb {R}} H(w(x,0))dx.
\end{equation}
Any function $w\in E$ is called the weak solution of the problem $(P_E)$ if for any $\phi \in E_V$
\begin{equation}\label{solp}
\int_{\mathbb {R}^2_+} \nabla w.\nabla \phi dxdy+\int_{\mathbb {R}}  w(x,0)\phi(x,0) V(x) dx-\int_{\mathbb {R}}h(w(x,0))\phi(x,0) dx=0.
\end{equation}
It is clear that critical points of $I$ in $E_V$ correspond to the critical points of $J$ in $H^{\frac{1}{2}}_0(\mathbb {R})$.
Thus if $w$ solves $(P_E)$ then $u= \textrm{trace}(w)=w(x, 0)$ is the solution of problem $(P)$ and vice versa.

\noindent With this introduction we state the main result of this section.
\begin{thm}\label{mht0}
Suppose $(h1)-(h5)$ are satisfied. Then the problem $(P)$ has a weak solution, $u$. If $V\in C^{1,\gamma}_{\rm loc}(\mathbb R)$, with $\gamma\in (0,1)$, then $u\in C^{2}(\mathbb R)$.
\end{thm}
\subsection{Mountain-pass solution}
\noindent We will use {the mountain pass lemma} to show the existence of solution in critical case. {The assumption ${\bf(V)}$ implies that $u\mapsto \int_{\mathbb R} |u|^q dx$ is weakly continuous for $q\in [2,\infty)$}. Next we have the following:
\begin{lem}
Assume that the conditions $(h1)-(h5)$ hold. Then $I$ satisfies  the mountain pass geometry around $0$.
\end{lem}
\begin{proof}
Using {assumption $(h2)$}, we get
\begin{equation*}
H(s)\geq C_1 |s|^\mu-C_2
\end{equation*}
for some $C_1, C_2>0$ and $\mu>2$. Hence for function $w\in E_V$ with compact support, we get
\begin{equation*}
I(tw)\leq \frac{t^2}{2} \|w\|^2-C_1 t^\mu \int_\mathbb R |w(x, 0)|^\mu dx+{C_2|\;\textrm{supp}\; w(x, 0)|}.
\end{equation*}
Hence $I(tw)\rightarrow -\infty$ as $t\rightarrow \infty$.
Next we will show that there exists $\alpha, \rho>0$ such that $I(w)>\alpha$ for all $\|w\|<\rho$. {From $(h1)$},
for $\epsilon>0, r>2$ there exists $C_1>0$ such that
\begin{equation*}
|H(s)|\leq \frac{\lambda_1-\epsilon}{2}s^2+C_1|s|^r(e^{(1+\epsilon)s^2}-1).
\end{equation*}
Hence, using {H\"older's} inequality, we get for $t>1$ and $t'=\frac{t}{t-1}$,
\begin{eqnarray*}
\int_{R}|H(w(x,0))|&\leq& \frac{\lambda_1-\epsilon}{2}\int_\mathbb R |w(x, 0)|^2 dx+C_1\int_{\mathbb R}|w(x, 0)|^r\left(e^{(1+\epsilon)w(x, 0)^2}-1\right)dx\\
&\leq& \frac{\lambda_1-\epsilon}{2}\|w\|^2_{L^2(\mathbb R)}+C_2\|w\|^r_{L^{t'}(\mathbb R)}\left(\int_\mathbb R \left(e^{(1+\epsilon)tw(x, 0)^2}-1\right)dx\right)^{\frac{1}{t}}\\
&\leq& \frac{\lambda_1-\epsilon}{2}\|w\|^2_{L^2(\mathbb R)}+C_2\|w\|^r_{L^{t'}(\mathbb R)}\left(\int_\mathbb R \left(e^{(1+\epsilon)t\|w\|^2\left(\frac{w(x, 0)}{\|w(x, 0)\|}\right)^2}-1\right)dx\right)^{\frac{1}{t}}.
\end{eqnarray*}
Now {let $w$ such that} $\|w\|\leq \rho$ for sufficiently small $\rho$ and $t$ close to 1 such that $(1+\epsilon)t\|w\|^2<\pi$. Then using Moser-Trudinger inequality in \eqref{newcor}, we get
\begin{eqnarray*}
I(w)&\geq& \frac{1}{2}\|w\|^2-\frac{\lambda_1-\epsilon}{2}\|w\|^2_{L^2(\mathbb R)}-C_3\|w\|^r_{{L^{t'}(\mathbb R)}}\\
&\geq& \frac{1}{2}\left(1-\frac{\lambda_1-\epsilon}{\lambda_1} \right)\|w\|^2-C_3\|w\|^r_{{L^{t'}(\mathbb R)}}.
\end{eqnarray*}
Hence {there exists} $\alpha>0$ such that $I(w)>\alpha$ for all $\|w\|\leq \rho$ for sufficiently small $\rho$.
\end{proof}

\noi Next we show the boundedness of Palais-Smale {sequences}.
\begin{lem}\label{psbo}
Every Palais-Smale sequence of $I$ is bounded in $E_V$.
\end{lem}
\begin{proof}
Let $\{w_k\}$ be a $(PS)_c$ sequence, {that is}
\begin{equation}\label{plsm}
I(w_k)=c+o(1) \;\;\textrm{and}\;\; I^{\prime}(w_k)=o(1).
\end{equation}
{Then},
\begin{eqnarray*}
\frac{1}{2}\|w_k\|^2-\int_{\mathbb R} H(w_k(x, 0))dx=c+o(1)\; \text{and}\;
\|w_k\|^2-\int_\mathbb R  h(w_k(x, 0))w_k(x, 0)dx=o({\|w_k\|}).
\end{eqnarray*}
Therefore,
\begin{equation*}
\left(\frac{1}{2}-\frac{1}{\mu}\right)\|w_k\|^2-\frac{1}{\mu}\int_\mathbb R \left(\mu H(w_k(x, 0))-h(w_k(x, 0))w_k(x, 0)\right) dx=c+o({\|w_k\|}).
\end{equation*}
Using assumption $(h2)$, we get $\|w_k\|\leq C$ for some $C>0$.
\end{proof}
\noindent We have the following version compactness Lemma {that is derived from the Vitali's convergence theorem}:
\begin{lem}\label{lem2.4}
For any $(PS)_c$ sequence $\{w_k\}$ of $I$, there exists $w_0 \in E_V$ such that, up to subsequence,  $ h(w_k(x, 0)) \rightarrow   h(w_0(x, 0)) $ in $L^{1}_{loc}(\mathbb R)$ and  $ H(w_k(x, 0))  \rightarrow  H(w_0(x, 0))$ in $L^1(\mathbb R)$.
\end{lem}
\begin{proof}
 From Lemma \ref{psbo}, we get that the sequence $\{w_k\}$ is bounded in $E_V$. Therefore, up to subsequence, $w_k\rightharpoonup w_0$ weakly in {$E_V$}, for some $w_0\in {E_V}$. Also from equation \eqref{plsm}, we get $C>0$ such that
\begin{equation*}
\int_{\mathbb R}h(w_k(x, 0))w_k(x, 0)dx \le C \;\textrm{and}\;\int_{\mathbb R}H(w_k(x, 0))dx \le C.
\end{equation*}
Now using {Vitali's} convergence theorem, we get
\begin{equation}\label{prt1}
h(w_k(x, 0))\rightarrow h(w_0(x, 0)) \;\textrm{in}\; L^1_{\textrm {loc}}(\mathbb R).
\end{equation}
Now to show second part of the above Lemma,  {$(h3)$} and generalized Lebesgue dominated convergence theorem we get
\begin{equation}\label{prt21}
H(w_k(x, 0))\rightarrow H(w_0(x, 0)) \;\textrm{in}\; L^1_{\textrm{loc}}(\mathbb R).
\end{equation}
Now { by assumption $(h1)$ for $R$, $A>0$},
\begin{equation*}
\int_\myText{${|x|>R}$\\${|w_k|>A}$}H(w_k(x, 0))dx \leq \frac{{C}}{A}\int_\myText{${|x|>R}$\\${|w_k|>A}$} |w_k(x, 0)|^{3}  dx+\frac{{C}}{A}\int_\myText{${|x|>R}$\\ ${|w_k|>A}$}h(w_k(x, 0))w_k(x, 0) dx.
\end{equation*}
Since  sequence $w_k$ is bounded, for any $\delta>0$, we can choose $A$ sufficiently large such that
\begin{equation}\label{prt22}
\int_\myText{${|x|>R}$\\${|w_k|>A}$}H(w_k(x, 0)) dx<\frac{\delta}{2} .
\end{equation}
Next, note that
\begin{eqnarray*}
\int_\myText{${|x|>R}$\\${|w_k|\leq A}$}H(w_k(x, 0))dx &\leq&C_A\int_\myText{${|x|>R}$\\${|w_k|\leq A}$}|w_k(x, 0)|^{2} dx\\
&\leq&2C_A \int_\myText{${|x|>R}$\\${|w_k|\leq A}$}|w_k-w_0|^{2}dx+ 2 C_A \int_\myText{${|x|>R}$\\${|w_k|\leq A}$}|w_0|^{2} dx.
\end{eqnarray*}
Using the compact embedding of $E_V$ into $L^{r}(\mathbb R), r\ge 2$, we can choose $R$ such that
\begin{equation}\label{prt23}
\int_\myText{$|x|>R$\\$|w_k|\leq A$}    H(w_k(x, 0)) dx<\frac{\delta}{2}.
\end{equation}
hence combining equations \eqref{prt21}, \eqref{prt22}, \eqref{prt23}, the proof follows.
\end{proof}

\noindent We {use} the following version of the ''sequence of Moser functions concentrated on the boundary'' {in the spirit of \cite{adyadav}:}
\begin{lem}\label{mser}
There exists a sequence $\{\phi_k\}\subset {E_V}$ satisfying
\begin{enumerate}
\item $\phi_k\geq 0$, $\textrm{supp}(\phi_k)\subset B(0, 1)\cap \mathbb R^2_+$,
\item $\|\phi_k\|=1$,
\item $\phi_k$ is constant in $ B(0, \frac{1}{k})\cap \mathbb R^2_+$,  and {$\phi_k^2=\frac{1}{\pi}\log k+O(1)$.}
\end{enumerate}
\end{lem}
\begin{proof}
Let \begin{equation}\label{moserfn}
 \psi_{k}(x,y) = \frac{1}{\sqrt{2\pi}}\left\{
\begin{array}{lr}
\sqrt{\log k}& 0\leq \sqrt{x^2+y^2}\leq {\frac{1}{k}}\\
 \frac{\log {\frac{1}{\sqrt{x^2+y^2}}}}{\sqrt{\log k}} & \frac{1}{k}\leq \sqrt{x^2+y^2} \leq 1\\
 0 & \sqrt{x^2+y^2}\geq 1.
\end{array}
\right.
\end{equation}
Then $\displaystyle \int_{\mathbb R^2}|\nabla \psi_k|^2 dxdy=1$ and $\displaystyle \int_{\mathbb R^2}|\psi_k|^2 dxdy={O(\frac{1}{\log k})}$. Let $\displaystyle \bar{\psi}_k=\psi_k|_{\mathbb R^2_+}$ and $\displaystyle \phi_k=\frac{\bar \psi_k}{\|\bar \psi_k\|}$. Then $\phi_k\geq 0$ and $\|\phi_k\|=1$. Also $\displaystyle \int_{\mathbb R^2_+}|\nabla \bar \psi_k|^2 dx dy=\frac{1}{2}$ and $\displaystyle \int_{\mathbb R}|\bar \psi_k|^2 dx dy={O(\frac{1}{\log k})}$. Therefore $\displaystyle\phi_k^2=\frac{1}{\pi}\log k+{O(1)}$.
\end{proof}
\noindent Define $\Gamma=\{\gamma\in C([0, 1]; {E_V}):\gamma(0)=0 \;\textrm{and}\;I(\gamma(1))<0\}$ and mountain pass level as
\begin{equation}\label{eq3.10}
c=\displaystyle \inf_{\gamma \in \Gamma } \displaystyle \max_{t\in[0, 1]}I(\gamma(t)).
\end{equation}
\begin{lem} Let $c$ be defined as in \eqref{eq3.10}. Then
$c<\frac{\pi}{2}$.
\end{lem}
\begin{proof}
We prove by contradiction. Suppose $c\geq \frac{\pi}{2}$.  Then we have
\begin{equation}\label{contra}
\displaystyle \sup_{t\geq 0}I(t\phi_k)=I(t_k\phi_k)\geq \frac{\pi}{2}
\end{equation}
where functions $\phi_k$ are given by Lemma \ref{mser}. From equation \eqref{contra}, we get
\begin{equation}\label{tkgro}
{t_k^2\geq \pi}.
\end{equation}
 Now as $t_k$ is point of maximum, we get $\frac{d}{dt}I(t\phi_k)|_{t=t_k}=0$. Therefore
\begin{equation}\label{derphi}
{t_k^2\|\phi_k\|^2=\int_{\mathbb R}  h(t_k\phi_k)t_k\phi_kdx.}
\end{equation}
Now we estimate the right hand side of equation \eqref{derphi} using assumption $(h4)$ as
\begin{eqnarray}\label{critical-estimate}
{t_k^2=}\int_{\mathbb R}  h(t_k\phi_k)t_k\phi_k dx&\geq&\int_{B(0,\frac{1}{k})}  h(t_k\phi_k)t_k\phi_kdx\nonumber\\
&\geq&{\frac{2}{k} h(t_k\phi_k(0))t_k\phi_k(0)}\nonumber\\
&\geq&{Ce^{\frac{1}{\pi}(t_k^2-\pi){(\log k+O(1))}}}\;\; \frac{h(t_k\phi_k(0))t_k\phi_k(0)}{e^{t_k^2\phi^2_k(0)}}.
\end{eqnarray}
Now {since $t_k$ is bounded we have that $t_k^2\rightarrow \pi$. Thus, \eqref{critical-estimate} together with assumption $(h4)$ contradict equation \eqref{tkgro}}.
\end{proof}
\noindent Next we prove Theorem \ref{mht0} using the above Lemmas.\\
\noindent \textbf{Proof of Theorem \ref{mht0}:} Using Lemma \ref{psbo}, we get a bounded $(PS)_c$ sequence. So there exists $w_0\in E_V$ such that, up to subsequence, $w_k\rightharpoonup w_0$ in $E_V$ and {$w_k(x,0)\rightarrow w_0(x,0)$ a.e. in ${\mathbb R}$}. We first prove that $w_0$ solves the problem, then we show that $w_0$ is non zero. From Lemma \ref{psbo} and equation \eqref{plsm}, there exists $C>0$ such that
\begin{equation*}
\int_{\mathbb R}h(w_k(x, 0))w_k(x, 0)dx{\leq }C \;\textrm{and}\;\int_{\mathbb R}H(w_k(x, 0))dx{\leq}C.
\end{equation*}
Now from Lemma \ref{lem2.4}, we get $h(w_k)\rightarrow h(w_0)$ in $L^1_{loc}(\mathbb R)$. So for $\psi\in C_c^\infty$ the equation \eqref{solp} holds. Hence from density of {$C_c^\infty({\mathbb R}^2_+)$} in $E_V$,  $w_0$ is weak solution of $(P_E)$.\\

\noindent Next we claim that $w_0\not\equiv 0$. Suppose not. Then from Lemma \ref{lem2.4}, we get $H(w_k(x,0))\rightarrow 0$ in $L^1({\mathbb R})$. Hence from equation \eqref{plsm},
we get $\frac{1}{2}\|w_k\|^2\rightarrow c$ as $k\rightarrow \infty$ or $\|w_k\|^2\leq \pi-\epsilon$ for some $\epsilon>0$. Let $0<\delta <\frac{\epsilon}{\pi}$ and $q=\frac{\pi}{(1+\delta)(\pi-\epsilon)}>1$. Using $sh(s)\leq C(e^{(1+\delta)s^2}-1)$ for some $C>0$ large enough and $(e^s-1)^q\leq (e^{sq}-1)$ for $q\geq 1$ and Moser- Trudinger inequality \eqref{newcor} we get
\begin{eqnarray*}
\int_{\mathbb R}| h( w_k)w_k|^q  dx &\leq& C\int_{\mathbb R}(e^{q(1+\delta)w_{k}^{2}}-1)dx\\
&\leq& C\int_{\mathbb R}(e^{q(1+\delta)\|w_k\|^2\frac{w_k^2}{\|w_k\|^2}}-1) dx <\infty.
\end{eqnarray*}
Therefore by $\int_{\mathbb R} h(w_k) w_k  dx \rightarrow 0$ and from equation \eqref{plsm}, we get $\displaystyle \lim_{k}\|w_k\|^{2}=0$, which is a contradiction. Hence $w_0$ is a nontrivial solution of the problem $(P_E)$. Now by theorem 5.2 of \cite{XcT}, we get $w_0\in L^{\infty}_{loc}(\mathbb R^{2}_{+})$.\\

\noindent To show the positivity and regularity of the solution (in case $V\in C^{1,\gamma}_{\rm loc}$), we take the cylinder $\mathcal{C}_a= (-a,a)\times (0,\infty)$ for $a>0$. Then $w_0$ satisfies the elliptic problem
$$ \quad {\left\{
\begin{array}{rrll}
 \quad  -\Delta v &= &0, \;v\ge 0\; \; \text{in}\;\mathcal{C}_a\\
  v &=& w_0 \; \text{on}\; \{-a,a\}\times (0,\infty),\\
  \frac{\partial v}{\partial \nu}&=&h(w_0)-V(x)w_0 \; \text{on} \; (-a,a)\times \{0\}\\
\end{array}
\right.}
$$
Now by taking odd extension to the whole cylinder $(-a,a)\times \mathbb R$ as in \cite{XcT} and noting that, $w_0, h(w_0)\in L^{p}(\mathcal{C}_a)$ we get $w_0\in C^{2}(\overline{\mathcal{C}})$ for some $\alpha$. Therefore, $u(x)=w_0(x,0)\in C^{2}({\mathbb{R}}).$ The positivity of the solution follows from Lemma 4.2 of \cite{XcT}.

\begin{rem}
We remark that Theorem \ref{mht0} holds for the weighted problem
 \[  (-\Delta)^{1/2}u +  u= K(x) h(u) \;\;\textrm{in}\;\; \mathbb{R}.\]
 with $h(u)\sim |u|^p, p>1$ is super-quadratic near $0$ and $K(x)$ satisfies the assumption { introduced in }\cite{mrsq}:
 \textbf{(K)} $K\in L^{\infty}(\mathbb R)\cap C(\mathbb R)$. Further, for  any sequence $\{A_n\}$ of Borel sets of $\mathbb R$ with $|A_n|\le R$ for some $R>0$,
 \begin{equation*}
 \displaystyle \lim_{r\rightarrow \infty}\int_{A_n\bigcap \mathbb R\setminus[-r, r]}K(x)dx=0 \;\textrm{uniformly for all}\; n.
 \end{equation*}
In this case the embedding $E_{1}$ into $L^{1}(\mathbb R; K)$ is compact for $r\in (2,\infty)$. If $K\in L^1(\mathbb R)$ then the embedding is compact for $r\in [2,\infty)$.  \\
\end{rem}
\begin{rem}\label{convex-concave}
We remark that the methods and ideas applied {in the present section}  can be used to show the existence of two solutions (for small $\lambda$) for the following problem with convex-concave nonlinearity:
\begin{equation}
(-\Delta)^{\frac{1}{2}} u+ V(x) u= \lambda u^{q} + h(u) , u> 0 \;\textrm{in } \;\mathbb R,
\end{equation}
where $0<q<1$.

\noi The harmonic extension $w(x,y)$ of $u(x)$ satisfies the problem:
\begin{equation}
\begin{array}{rllll}
\Delta w = 0 \;\textrm{in } \;\mathbb \mathbb{R}^{2}_{+},\\
\frac{\partial w}{\partial y}= \la u^q + h(u) -V(x) u\; \text{in}\; \mathbb R.
\end{array}
\end{equation}
The variational functional $J_\la: E \rightarrow \mb R$ is defined as
\[J_\la(w) = \frac{1}{2}\|w\|^2 -\frac{\la}{q+1}\int_\mb R |{w(x,0)}|^{q+1}dx +\int_\mb R H(w(x,0))dx.\]
Then it is not difficult to show that $J_\la$ satisfies
\[J_\la (w) \ge  \left(\frac{1}{2}-\frac{\la_1 -\e}{\la_1}\right) \|w\|^2-\frac{\la}{q+1} \|w\|^{q+1}- C \|w\|^{p}\]
for some $p>2$. So by taking $\|w\|$ small {enough} we can show that there exists $\rho, R_0, \la_0$  such that for all $\la<\la_0$
\[J_\la (w) \ge \rho>0 \;\; \text{on}\; \|w\| = R_0.\]
Also, for a fixed $\phi $ with compact support in $\mb R^{2}_{+}$, for all small $t$
$J_\la (t\phi) <0.$
So, we can consider the minimization problem: \[ \min_{\|w\|\le R_0} J_\la(w).\]
This minimum is clearly negative and so one can follow  Lemma 3.5 and Theorem 3.1 to show the existence of a solution $w_\la$. Also, $w_\la$ is strict local minimum of $J_\la$ for $\la \in (0,\la_0)$.

\noindent To show the existence of second solution, we can translate the problem to the origin. {In other words, a second  solution} can be obtained as $v_\la = w_\la + v$ where $v$ satisfies
\begin{equation*}\label{Et}
\left\{\begin{array}{rl}
 -{\Delta} v&=0\;,\;v\ge 0\;\;\quad \quad \quad \quad\textrm{in}\; \mb{R}^{2}_{+} ,\\
  \frac{\partial v}{\partial \nu}&=\lambda g(v+w_\la)+h(v+w_\la)-\lambda g(w_\la)-h(w_\la)- {V}(v+w_\la)+V w_\la \;\; \textrm{on}\;\; \mb R
\end{array}
\right.
\end{equation*}
where $g(s)=s^{q}$.
The corresponding functional $\tilde{J}_{\la} :E \rightarrow \mb R$ is defined by
\[\tilde{J}_{\la}^{t}(v)=\frac{1}{2}\|v\|^2 - \la \int_\mb R \tilde{G}(v(x,0)) -\int_{\mb R} \tilde{H}(v(x,0))\]
where $\tilde{G}(v)=\int_{0}^{v} (g(s+w_\la)-g(w_\la)) ds$ and $\tilde{H}(v)=\int_{0}^{v} (h(s+w_\la)-h(w_\la)) ds$.
Since $w_\la$ is a strict local minimum and $\displaystyle \lim_{t\rightarrow{\infty}}\tilde {J}_\lambda(tv)=-\infty$ for fixed $v$,  we can fix $e\in E$ such that $\tilde {J}_\lambda(e)<0$. Let
\[
\Gamma=\{\gamma:[0, 1]\rightarrow H^1_{0, L}(\mathcal C): \gamma \;\textrm{continuous on}\; [0, 1], \gamma(0)=0, \gamma(1)=e\}.\]
Define the mountain pass level $\rho_0=\displaystyle\inf_{\gamma \in \Gamma}\displaystyle \sup_{t\in [0, 1]}\tilde {J}_\lambda(\gamma(t))$.
Now using the growth assumptions on $g$ and $h$, one can show as in Lemma 3.5 that there exists a Palais-Smale sequence below the mini-max level $\rho_0$ and the existence of { a positive solution} $v$ follows { similar arguments as} in the proof of theorem 3.1.
\end{rem}

\section{Kirchhoff equation with critical nonlinearity}
\setcounter{equation}{0}
\noindent In this section we study the problem
\[ (Q)\;\quad \; m\left(\int_{\mathbb R}|(-\Delta)^{\frac{1}{2}}u|^2 dx+\int_{\mb R} u^2 V dx\right)\left((-\Delta)^{\frac{1}{2}} u+ V(x) u\right)= f(u)\;\; \text{in}\; \mathbb R,\]
where $V\in C^{0,\gamma}_{\rm loc}(\mathbb R)$ with $0<\gamma<1$ verifies ${\bf (V)}$ and $m:\mathbb R^+ \rightarrow \mathbb R^+$ and $f:\mathbb R \rightarrow \mathbb R$ are continuous functions that satisfy the following assumptions:
\begin{enumerate}
\item[\bf (m1)] There exists $m_0>0$ such that $m(t)\ge m_0$ for all $t\ge 0$ and
\[M(t+s) \ge M(t)+M(s)\; \text{for all} \; s,t \ge 0,\]
where $M(t)=\int_0^t m(s) ds$, the primitive of $m$.
\item[\bf (m2)] There exists constants $a_1, a_2>0$ and $t_0>0$ such that for some $\sigma \in \mathbb R$
\[m(t)\le a_1 {+}a_2 t^\sigma,\; \text{for all}\; t \ge t_0.\]
\item[\bf (m3)] $\frac{m(t)}{t}$ is {decreasing} for $t>0.$
\end{enumerate}
A typical example of  a function satisfying ${\bf (m1)-(m3)}$ is $m(t)=m_0 +at$, where $m_0>0$  and $a\ge 0$. Also, from ${\bf (m3)}$ one can deduce that
\begin{equation}\label{4eq1}
\frac{1}{2}M(t)-\frac{1}{4}m(t) t \ge 0\;\; \text{for all} \; t\ge 0.
\end{equation}

\noindent The nonlinearity $f(t)$ satisfies
\begin{enumerate}
\item[\bf (f1)]$f\in C^1(\mathbb R), f(t)=0 $ for $t\le 0$, $f(t)>0$ for $t>0$  and $f(t)$ satisfies for any $\varepsilon>0$, $\displaystyle\lim_{t\to\infty} f(t)e^{-(1+\varepsilon) t^2}=0. $
 Moreover, $\displaystyle\lim_{t\to 0}\frac{f(t)}{t^3}=0$ and $\displaystyle \frac{f(t)}{t^\theta}$ is increasing in $(0,\infty)$ for some $\theta>3$.
\item[\bf (f2)] There exist positive constants $t_0, K_0$ such that
\[F(t)\le K_0 f(t), \;\; \text{for all}\; t\ge t_0, \; \; \text{where}\; F(t)=\int_{0}^{t} f(s) ds.\]
\item[\bf(f3)] $\displaystyle \lim_{t\rightarrow \infty} t {f}(t) e^{-t^2} =\infty.$
\end{enumerate}
\begin{rem}\label{examples}
A prototype examples of $f$ satisfying {\bf (f1)}-{\bf (f3)} are $t^pe^{t^2}$ with $p\geq \theta>3$ and $t^p(e^{t^\beta}-1)e^{t^2}$ with $p\geq \theta$ and $\beta\in (0, 2)$.
\end{rem}
Let $w(x,y)$ be the harmonic extension of $u(x)$. Then $w(x,y)$ satisfies the problem
\begin{equation}
\label{4eq2}
\begin{array}{rlll}
\Delta w &=0 \; \; \text{in}\; \mathbb R^{2}_{+}\\
\frac{\partial w}{\partial y}&= -V(x) u+ \frac{f(u)}{m(\|u\|_{X}^2)}\; \quad \text{on}\; \mathbb R.
\end{array}
\end{equation}
From the definition of $\mathcal{E}_{\frac{1}{2}}$, we have
\[\mathcal{E}_{\frac{1}{2}}(u)=w, ;\; \; \int_{\mathbb R}|(-\Delta)^{{\frac{1}{4}}}u|^2 dx= \int_{\mathbb R^{2}_{+}} |\nabla w(x,y)|^2 dx dy\; \text{and}\; w(x,0)=u(x).\]
Therefore, the problem in \eqref{4eq2} is equivalent to
\begin{equation}\label{4eq3}
\begin{array}{rlll}
-m\left(  \| w\|^2 \right)\Delta w &=0 \; \; \text{in}\; \mathbb R^{2}_{+}\\
-m\left(  \| w\|^2 \right)\left(\frac{\partial w}{\partial y}+V(x)w(x,0)\right)&= f(w(x,0))\; \quad \text{on}\; \mathbb R
\end{array}
\end{equation}
where $\Vert w\Vert:=\left(\int_{{\mathbb R}^2_+}\vert\nabla w\vert^2dx+\int_{\mathbb R}V(x)w(x,0)^2dx\right)^{\frac{1}{2}}$.
\begin{defi}
We say that $w\in E_V$ is a weak solution of \eqref{4eq3} if
\[ m(\| w\|^2 )\left(\int_{\mathbb R^{2}_{+}} \nabla w \nabla \phi dxdy +\int_{\mathbb R} V(x) w(x,0) \phi(x,0) dx\right) -\int_{\mathbb R} f(w(x,0))\phi(x,0) dx\]
holds for all $\phi \in E_V$.
\end{defi}
\noindent The variational functional corresponding to \eqref{4eq3} is defined as $J: E_V \rightarrow \mathbb R$ as
\begin{equation}\label{4eq4}
J(w)=\frac{1}{2} M\left(\| w\|^2\right) -\int_{\mathbb R} F(w(x,0))dx.
\end{equation}
Then the {the trace of }critical points of the functional $J$ are weak solutions {to} $(Q)$. We prove the following theorem in this section.
\begin{thm}\label{thm4.1}
Suppose {\bf (m1)}-{\bf (m3)} and {\bf (f1)}-{\bf (f3)} are satisfied. Then the problem $(Q)$ has a positive weak solution.
\end{thm}
\noindent We prove this theorem by using the mountain pass lemma. In the next few lemmas we study the mountain pass structure and properties of Palais-Smale sequences to the functional $J$. Our proofs closely follow \cite{Gf}.

\begin{lem}\label{lem4.1}
Assume the conditions ${\bf (m1)}$, ${\bf (f1)}-{\bf (f3)}$ hold. Then $J$ satisfies mountain-pass geometry {around $0$}.
\end{lem}
\begin{proof} From the assumptions, ${\bf (f1)}-{\bf (f3)}$, for $\e>0$, $r>2$, there exists $C>0$ such that
\[|F(t)| \le \e |t|^2 + C |t|^r (e^{(1+\epsilon)|t|^{2}}-1),\;\; \text{for all}\; t\in  \mb R.\]
Therefore, using Sobolev and H\"{o}lder inequalities, for $w\in H^{1}(\mb R^{2}_{+})$, we get
\begin{align*}
\int_{\mb R} F(w(x,0))~ dx &\le \e\int_{\mb R} |w(x,0)|^2  +C \int_{\mb R} |w(x,0)|^r (e^{(1+\epsilon)|w(x,0)|^{2}}-1) dx \\
&\le \e C_1 \|w\|^2 + C \|w\|_{{L^{2r}(\mathbb R)}}^{r} \left(\int_{\mb R} (e^{2(1+\epsilon)\|w(x,0)\|^{2} (\frac{w}{\|w\|})^{2}}-1) \right)^{1/2}\\
&\le \e C_1 \|w\|^2 + C_2 \|w\|_{{L^{2r}(\mathbb R)}}^r
\end{align*}
\noi for $\|w\|<R_1$, where $(1+\epsilon)^{1/2}R_1\leq \left(\frac{\pi}{2}\right)^{\frac{1}{2}}$, thanks to Moser-Trudinger inequality in \eqref{newcor}.
Hence
\[J(w) \ge \|w\|^2 \left(\frac{m_0}{2}-\e C_1 - C_2 \|w\|^{r-2}\right).\]
Since $r>2,$ we can choose $\e$, $0<R\leq R_1$ small such that $J(w)\ge \tau$ for some $\tau$ on $\|w\|=R$.\\
\noi Now by {${\bf (f1)}$ and ${\bf (f3)}$}, for $\theta>\max\{2,2(\sigma+1)\}$, there exist $C_1$, $C_2>0$ such that
 \begin{equation}\label{n3}
 F(t)\geq C_1 t^{\theta}- C_2\;\mbox{for all}\; t \ge 0
 \end{equation}
 and  condition $(m2)$ implies that {for all $t\geq t_0$}
\begin{equation}\label{n2}
M(t)\leq\left\{
 \begin{array}{lr}
 a_0+a_1t+\frac{a_2}{\sigma+1} t^{\sigma+1},\; \mbox{if}\; \sigma\ne -1,\\
 b_0+ a_1 t+a_2 \ln t\quad\quad\mbox{if}\; \sigma=-1,
 \end{array}
 \right.
\end{equation}
where $a_0= M(t_0) -a_1t_0- \frac{a_2}{\sigma+1} t_0^{\sigma+1}$ and $b_0= M(t_0) -a_1 t_0-a_2\ln t_0$. Now, choose a function $\phi_0 \in E_V$ with compact support, $\phi_0\geq 0$ and $\| \phi_0\|=1$. Then from \eqref{n3} and \eqref{n2}, for all $t\geq t_0$, we obtain
\begin{equation*}
J(t\phi_0)\leq \left\{
\begin{array}{lr}
\frac{a_0}{2}+\frac{a_1}{2}t^2+\frac{a_2}{2\sigma+2} t^{2\sigma+2}  - C_1 t^{\theta}\|\phi_0\|^{\theta}_{\theta}+C_2 ,\; \mbox{if}\; \sigma\ne -1,\\
\frac{ b_0}{2}+ \frac{a_1}{2} t^2 +\frac{a_2}{2} \ln t - C_1 t^{\theta}\|\phi_0\|^{\theta}_{\theta}+C_2 \;\;\;\quad\quad\mbox{if}\; \sigma=-1,
\end{array}
\right.
\end{equation*}
from which we conclude that $J(t \phi_0)\ra -\infty$ as $t\ra +\infty$ provided that $\frac{\theta}{2}>\max\{1, \sigma+1\}$.
Therefore, $J$ satisfies mountain-pass geometry near $0$.
\end{proof}

\begin{lem}\label{lem4.2}
Every Palais-Smale sequence of $J$ is bounded in $E_V$.
\end{lem}
\begin{proof}
Let $\{w_k\}\subset E_V$ be a Palais-Smale sequence for $J$ at level $c$, that is
\begin{equation}\label{n7a1}
{\frac{1}{2}M(\|w_k\|^2)-\int_{\mb R} F(w_k(x,0)) \rightarrow c}
\end{equation}
and for all $\phi \in E_V$
\begin{equation}\label{n7a2}
 \left|-m(\|w_k\|^2)\left(\int_{\mb R^{2}_{+}}  \na w_k \na \phi ~dx dy+\int_{\mb R} {w_k(x,0) \phi V(x)} dx\right) -\int_{\mb R} f(w_k(x,0)) \phi(x,0)~ dx\right| \le \e_k \|\phi\|
\end{equation}
where $\e_k\rightarrow 0$ as $k\rightarrow \infty.$
From {\bf (m3)}, {{\bf (f1)}},  \eqref{n7a1} and \eqref{n7a2}, we obtain { that there exists $C>0$ independent of $k$ such that}
 \begin{align*}
C+ {\e_k}\|w_k\|&\ge \frac{1}{2}M(\|w_k\|^2)-\frac{1}{\theta} m(\|w_k\|^2) \|w_k\|^2 \\
&\quad\quad-\int_{\mb R} \left(F(w_k(x,0))-\frac{1}{\theta}f(w_k(x,0))w_k(x,0)\right)~dx\\
&\geq \left(\frac{1}{4}-\frac{1}{\theta}\right) m(\|w_k\|^2) \|w_k\|^{2}.
\end{align*}
From this and taking $\theta >4$ as in {\bf (f1)}, we obtain the boundedness of the sequence.  \end{proof}

\noi Let $\ds \Gamma=\left\{\gamma\in C\left([0,1],H^{1}(\mb R^{2}_{+})\right):\gamma(0)=0,\; J(\gamma(1))<0\right\}$ and define the mountain-pass level
\begin{equation}\label{eq4.9}
 c_*=\inf_{\gamma\in \Gamma}\max_{t\in[0,1]}J(\gamma(t)).
 \end{equation}
Then we have,

\begin{lem}\label{lem4.3}
Let $c_*$ be defined as in \eqref{eq4.9}. Then $\ds c_*<\frac{1}{2}M(\pi)$.
\end{lem}
\begin{proof} Let  {$\phi_k$} be the sequence of Moser functions as in Lemma \ref{mser}. Assume by contradiction that $c_*\geq\frac{1}{2}M(\pi)$. Then for each $k$, there exists $t_k$ such that
\begin{equation}\label{7a1}
\sup_{t>0}J(t\phi_k)=J(t_k\phi_k)=\frac{1}{2}M(\|t_k\phi_k\|^2)-\int_{\mb R} F(t_k\phi_k(x,0)) \geq  \frac{1}{2}M(\pi).
\end{equation}
From \eqref{7a1}, we see that $t_k$ is a bounded sequence as $J(t_k\phi_k)\rightarrow-\infty$ as $t_k\rightarrow\infty.$
Also using the fact that $M$ is monotone increasing and $F(t_k \phi_k(x,0))\geq 0$ in \eqref{7a1}, we obtain
\begin{equation}\label{7a2}
t_k^2\geq\pi.
\end{equation}
Now since $t_k$ is a point of maximum for one dimensional map $t\mapsto J(t\phi_k),$ we have
$\frac{d}{dt}J(t\phi_k)|_{t=t_k}=0.$ From this it follows that
\begin{align}\label{new7a2}
 {m(t_k^2)t_k^2=}m(t_k^2\|\phi_k\|^2)t_k^2\|\phi_k\|^2 &=\int_{\mb R} f(t_k\phi_k(x,0))t_k\phi_k(x,0)\geq \int_{-{\frac{1}{k}}}^{{\frac{1}{k}}} f(t_k\phi_k(x,0))t_k\phi_k(x,0)\notag\\
 &= \frac{2}{k} t_k \phi_k(0) f(t_k\phi_k(0)).
\end{align}
{Then from the above inequality and \eqref{7a2}, it follows that $t_{k}^{2} \rightarrow \pi$}. Now as in Lemma \ref{mser},  {$(f3)$} and
\eqref{new7a2} gives the required  contradiction. Hence
$c_*<\frac{1}{2}M(\pi)$.\end{proof}

\noi Now we have the following version of Lion's Lemma. The proof here is an {adaptation} of Lemma 2.3 of \cite{yang}.
\begin{lem}\label{lions}
Let $\{w_k\}$ be a sequence in $E_V$ with $\|w_k\|=1$ and $w_k \rightharpoonup w_0$ weakly in $E_V$. Then for any $p$ such that $1<p<\frac{1}{1-\|w_0\|^2}$, we have
\[\sup_k \int_{\mb R} \left(e^{\al p (w_{k}(x,0))^2}-1\right) dx <\infty, \quad \text{for all}\; \; 0<\al<\pi.\]
\end{lem}

\begin{proof}
First we note that from the Young inequality, for $\frac{1}{\mu}+\frac{1}{\nu}=1,$
\begin{equation}\label{eq4.6}
e^{s+t}-1 \le \frac{1}{\mu} (e^{\mu s}-1)+ \frac{1}{\nu}(e^{\nu t}-1).
\end{equation}
Now using the inequality $w_{k}^{2}\le (1+\e) (w_k-w_0)^2 + C(\e) w_{0}^{2}$ and \eqref{eq4.6}, we get
\begin{align}
e^{\al p w_{k}^{2}}-1 &\le e^{\al p \left((1+\e)(w_k-w_0)^2+C(\e) w_{0}^{2}\right)}-1\\
&\le \frac{1}{\mu} \left(e^{\al p \mu \left((1+\e)(w_k-w_0)^2\right)}-1\right) + \frac{1}{\nu}\left(e^{\al p\nu \left(C(\e) w_{0}^{2}\right)}-1\right).
\end{align}
Now using the fact that $\|w_k -w_0\|^2= 1- \|w_0\|^2 +o_k(1)$, we get
\[
\al p \mu \left((1+\e)(w_k-w_0)^2\right)= \left(\al p \mu (1+\e) (1-\|w_0\|^2+o_k(1))\right)  \left(\frac{(w_k-w_0)}{\|w_k-w_0\|}\right)^2.\]
Hence for any $1<p<\frac{1}{1-\|w_0\|^2}$, and {$\e>0$ small enough} and $\mu>1$ close to $1$, {we have that}
\[\al p \mu (1+\e) (1-\|w_0\|^2) <\pi\]
and the proof follows from {\eqref{newcor}} and Lemma \ref{orlicz-property}.
\end{proof}

{From the fact that $\frac{f(t)}{t^3}$ is increasing (since $\theta\geq 3$), we deduce easily the following lemma.}

\begin{lem}\label{lem7.3}
{Let} condition ${\bf (f1)}$ holds. Then, $ sf(s)-4F(s)$ is increasing for $s\ge0$. In particular $sf(s)-4F(s)\geq 0$ for all $s\ge 0$.
\end{lem}
\noi Now we define the Nehari manifold associated to the functional $J$, as
\[ \mc{N}:=\{0\not\equiv u\in {E_V}:\langle J^{\prime}(u),u \rangle=0\}\]
and let $\ds b:=\inf_{u\in \mc{N}} J(u)$. Then we need the following Lemma that compare  $c_*$ and $b$.

\begin{lem} $c_*\leq b $.
\end{lem}
\proof Let $u\in \mc{N}$, define $h:(0,+\infty)\rightarrow \mb R$ by $ h(t)=J(tu)$. Then
\[h^{\prime}(t)=\langle J^{\prime}(tu),u\rangle=m(t^2\|u\|^2)t \|u\|^2-\int_\mb R f(tu)u~dx \; \text{for all} \; t>0.\]
Since $\langle J^{\prime}(u),u\rangle=0$, we have {\small
\[ h^{\prime}(t)=\|u\|^{{4}} t^{3}\left(\frac{m(t^2\|u\|^2)}{t^2\|u\|^2}-\frac{m(\|u\|^2)}{\|u\|^2} \right)+t^{3}\int_{\mb R}\left(\frac{f(u)}{u^{3}}-\frac{f(tu)}{(tu)^{3}}\right)u^{4}dx.\]
{From ${\bf (m3)}$ and ${\bf (f1)}$, we get} $h^{\prime}(1)=0$, $h^{\prime}(t)\geq 0$ for $0<t<1$ and $h^{\prime}(t)<0$ for $t>1$. Hence $J(u)=\ds \max_{t\geq0}J(tu)$. Now define $g:[0,1]\rightarrow {E_V}$ as  $g(t)=(t_0 u)t$, where $t_0$ is such that $J(t_0 u)<0$. We have $g\in \Gamma$ and therefore
\[ c_*\leq\max_{t\in[0,1]}J(g(t))\leq \max_{t\geq 0}J(tu)=J(u).\]
  Since $u\in \mc N$ is arbitrary, $c_*\leq b$ and the proof is complete.\qed\\

\noi {\bf Proof of Theorem \ref{thm4.1}:} Let $\{w_k\}\in {E_V}$ be a Palais-Smale sequence  of $J$ at level $c_*>0$. That is $J(w_k)\rightarrow c_*$ and $J^\prime (w_k) \rightarrow 0$. Then by Lemma \ref{lem4.2}, there exists $w_0\in {E_V}$ such that $w_k \rightharpoonup w_0$ weakly in ${E_V}$. Now we claim that $w_0$ is the required positive solution.\\
Now it is not difficult to see that $w_0\not \equiv 0.$ Indeed, if $w_0\equiv 0$. Then $\int_{\mb R} F(w_k(x,0))dx \rightarrow 0$ and hence
\[\frac{1}{2}M(\|w_k\|^2)\rightarrow c_*< \frac{1}{2} M(\pi).\]
Therefore, there exists $\ba$ such that $\|w_k\|^2 <\beta <\pi$ for all $k$ large. So we can find $q>1$ and close to $1$ so that $q \|w_k\|^2 <\pi.$ Now it is easy to show that {$\int_{\mathbb R}f(w_k(x,0))w_k(x,0)dx \rightarrow 0$}. Hence
\[o_k(1)=\langle J^\prime(w_k), w_k\rangle = m(\|w_k\|^2) \|w_k\|^2 -\int_\mb R f(w_k) w_k = m(\|w_k\|^2) \|w_k\|^2+o_k(1).\]
That is, $\|w_k\|\rightarrow 0$ and $J(w_k)\rightarrow 0,$  {which provides a contradiction}.

\noi To show the positivity of the solution,  we see that $\|w_k\|\rightarrow \rho_0>0$ (up to a subsequence). So $J^{\prime}(w_k)\rightarrow 0$  implies that $\text{for all} \; \phi \in E_V$, we have
\begin{equation}
 m(\rho_0^2) \left(\int_{\mb R^{2}_{+}} \nabla w_0\nabla \phi~ dx + \int_{\mb R} w_0(x,0) \phi(x,0) V(x){dx}\right) - \int_\mb R {f(w_0(x,0))\phi(x,0)} dx=0.
 \end{equation}
  That is $u_0(x)=w_0(x,0)$ satisfies the equation  $ (-\De)^{\frac{1}{2}} u_0 + V u_0=\frac{1}{m(\rho_0^2)}f(u_0)\;\mbox{in}\;\mb R$.
 Using the growth condition { in ${\bf (f1)}$ of $f$} and Trudinger-Moser inequality, we get $f(u_0)\in L^{p}_{loc}(\mb R)$ for all $1< p < \infty$. Therefore by regularity result in proposition 3.1, page 21 in \cite {XcT},  $u_0\in C^{1,\gamma}_{loc}(\mb R)$ and hence by strong maximum principle ({see} Lemma 4.2 of \cite{XcT}), we get $u_0>0$ in $\mb R$.\\

\noi \textbf{claim 1:} $\ds m(\|w_0\|^2)\|w_0\|^2\geq \int_\mb R f(w_0)w_0 dx$.
\proof The proof follows ideas in Lemma 5.1 of \cite{Gf}. For completeness, we give the details here.  Suppose by contradiction that $m(\|w_0\|^2)\|w_0\|^2< \int_\mb R f(w_0)w_0 ~dx$. That is, $\langle J^{\prime}(w_0),w_0\rangle<0.$
 Using {\bf(f1)} and Sobolev imbedding, we can see that $\langle J^{\prime}(tw_0),w_0\rangle>0$ for $t$ sufficiently small. Thus there exists $\sigma\in(0,1)$ such that $\langle {J}^{\prime}(\sigma w_0),w_0\rangle=0$. That is,  $\sigma w_0\in \mc N$. Thus according to  Lemma \ref{lem7.3} {and ${\bf (m3)}$},
 \begin{align*}
 c_*&\leq b \leq J(\sigma w_0)=J(\sigma w_0)-\frac{1}{4}\langle J^{\prime}(\sigma w_0),\sigma w_0\rangle\\
 &=\frac{M(\|\sigma w_0\|^2)}{2}-\frac{m(\|\sigma u_0\|^2)\|\sigma u_0\|^2}{4}+\int_\mb R \frac{(f(\sigma w_0)\sigma w_0-4F(
 \sigma w_0))}{4}\\
 &<\frac{1}{2}M(\| w_0\|^2)-\frac{1}{4}m(\|w_0\|^2)\| w_0\|^2+\frac{1}{4}\int_\mb R (f( u_0)u_0-4F(u_0)).
 \end{align*}
 By lower semicontinuity of norm and Fatou's Lemma, we get
 \begin{align*}
 c_*&<  \liminf_{k\rightarrow\infty}\frac{1}{2}\left(M(\|{w_k}\|^2)-\frac{1}{2}m(\|{w_k}\|^{2})\| {w_k}\|^{2}\right)+\liminf_{k\rightarrow\infty}\frac{1}{4}\int_\mb R [f( {w_k})u_k-4F({w_k})]dx\\
 &\leq \lim_{k\rightarrow\infty}[J({w_k})-\frac{1}{4}\langle J^{\prime}({w_k}),{w_k}\rangle]=c_*
 \end{align*}
 which is a contradiction and the claim 1 is proved.

\noi \textbf{Claim 2:} $J(w_0)=c_*$.
\proof Using $ \int_\mb R F({w_k(x,0)})\rightarrow \int_\mb R F({w_0(x,0)})$ and lower semicontinuity of norm we have $J(w_0)\leq c_*$.
 We are going to show that the case $J(w_0)<c_*$ can not occur.\\
 Indeed,  if $J(w_0)<c_*$ then $\|w_0\|^2<\rho_0^2.$ Moreover,
 \begin{equation}\label{7a3}
   \frac{1}{2}M(\rho_0^2)=\lim_{k\rightarrow\infty}\frac{1}{2}M(\|u_k\|^2)=c_*+\int_\mb R F(w_0)dx
 \end{equation}
  %$\frac{1}{n}M(\rho_0^n)=\lim_{k\rightarrow\infty}\frac{1}{n}M(\|u_k\|^n)=c_*+\int F(x,u_0)dx$\\
  which implies  $\ds \rho_0^2 = M^{-1}(2c_*+2\int_{\mb R} F(w_0)dx)$.
  Next defining ${v_k}=\frac{{w_k}}{\|{w_k}\|}$ and ${v_0}=\frac{w_0}{\rho_0}$, we have $v_k\rightharpoonup {v_0}$ in ${E_V}$ and $\|{v_0}\|<1$.

\noi On the other hand, by {claim 1} and Lemma \ref{lem7.3}, we have
\[  J(w_0)\ge \frac{M(\|w_0\|^2)}{2}-\frac{m(\|w_0\|^2)\|w_0\|^2}{4} +\int_\mb R \frac{(f(w_0)w_0-4 F(w_0)}{4}\geq 0.\]
Using this together with Lemma \ref{lem4.3} and the equality:
$2(c_*-J(w_0))=M(\rho_{0}^{2})-M(\|w_0\|^2)$,
we get
$M(\rho_{0}^2) \le 2 c_*+M(\|w_0\|^2)<M(\pi)+M(\|w_0\|^2)$.
Therefore by ${\bf (m1)}$
\begin{equation}\label{7a6}
\rho_{0}^{2}< M^{-1}\left(M(\pi)+M(\|w_0\|^2)\right)\le \pi+\|w_0\|^2.
\end{equation}
Since $\rho_{0}^{2}(1-\|v_0\|^2)=\rho_{0}^{2}-\|w_0\|^2$, from \eqref{7a6} it follows that
$ \rho_{0}^{2} (1-\|v_0\|^2) < \pi.$
Thus, there exists $\ba>0$ such that $ \|w_k\|^{2} < \ba < \frac{\pi}{1-\|v_0\|^2}$ for $k$ large. We can choose $q>1$ close to $1$ such that $q \|w_k\|^{2} \le  \ba < \frac{\pi}{1-\|v_0\|^2}$ and using {Lemma \ref{lions}}, we conclude that for $k$ large
\[ \int_\mb R \left(e^{q |w_k(x,0)|^{2}}-1\right) dx \le \int_\mb R \left(e^{\ba |v_k(x,0)|^{2}}-1\right) \le C.\]
Now by standard calculations, using H\"{o}lder's inequality and weak convergence of  $\{w_k\}$ to $w_0$, we get $\int_\mb R f(w_k) (w_k-w_0) \rightarrow 0$ as $k\rightarrow \infty$. Since $\langle J^\prime (w_k), w_k-w_0\rangle \rightarrow 0$, it follows that
\begin{equation} \label{na7new}
m(\|w_k\|^2) \int_{\mb{R}^{2}_{+}}  \na w_k (\na w_k-\na w_0) \rightarrow 0.
\end{equation}
%\noi On the other hand,
Now by weak convergence of $w_k$ and $m(t)>0$, we get $w_k\ra {w_0}$ strongly in ${E_V}$ and $J(w_0)=c_*$. This ends the proof of claim 2.

\noi Now by claim 2 and \eqref{7a3} we can see that $M(\rho_{0}^{2})=M(\|w_0\|^2)$ which implies that $\rho_{0}^{2}=\|w_0\|^2$. Hence, $w_0$ is a weak solution of \eqref{4eq3}. \qed

\noindent {\bf Acknowledgements:} The authors were funded by IFCAM (Indo-French Centre for Applied Mathematics) UMI CNRS under the project "Singular phenomena in reaction diffusion equations and in conservation laws".

 \end{document}